\documentclass[a4paper,onecolumn,oneside,11pt,leqno]{article}

\usepackage{amsmath, amsfonts,amssymb}
\usepackage{theorem, euscript,array,enumerate}
\usepackage{amsfonts,mathrsfs, latexsym}
\usepackage[dvips]{graphicx}
\usepackage[dvips]{graphics}
\usepackage[dvips]{epsfig}

\theoremheaderfont{\bfseries} \theoremstyle{plain}
\theorembodyfont{\upshape}
\newcommand{\esp}{\mathbb{E}} %esperance

\newcommand{\proba}{\mathbb{P}}

\theoremstyle{plain}

\theorembodyfont{\upshape\it}
\newtheorem{Thm}{\bf Theorem}[section]
\newtheorem{Pro}[Thm]{\bf Proposition}
\newtheorem{Lem}[Thm]{\bf Lemma}
\newtheorem{Cor}[Thm]{\bf Corollary}
\theorembodyfont{\upshape}

\newtheorem{Rem}[Thm]{Remark}

\newenvironment{proof}[1][{\it Proof.}]{\begin{trivlist}
\item[\hskip \labelsep {\bfseries #1}]}
{ {\quad}\hfill $\Box$\end{trivlist}\vskip
-0.2 cm}

\begin{document}
\thispagestyle{empty}
\title{Zero bias transformation and asymptotic expansions II :\\
the Poisson case}
\author{Ying Jiao\thanks{Laboratoire de probabilit\'es et mod\`eles al\'eatoires,
Universit\'e Paris 7; jiao@math.jussieu.fr}}
\date{\today}
\maketitle

\begin{abstract}
We apply a discrete version of the methodology in \cite{gauss} to obtain
a recursive asymptotic expansion for $\esp[h(W)]$  in terms of Poisson expectations, where $W$ is a sum of independent integer-valued random variables and $h$ is a polynomially growing function. We also discuss the remainder estimations.

\vspace{3mm}

\noindent MSC 2000 subject classifications: 60G50, 60F05.\\
Key words:  Poisson approximation, zero bias transformation, Stein's method, asymptotic expansions, discrete reverse Taylor formula.
\end{abstract}

\section{Introduction and main result}

It should be noted in the first place that the notation
used in this paper is similar as in \cite{gauss},
however, their meanings are different since we here
consider discrete random variables. Stein's method for
Poisson appoximation has been introduced by Chen
\cite{Chen1975}. Let $Z$ be an $\mathbb N$-valued
random variable ($\mathbb N$-r.v.), then $Z$ follows
the Poission distribution  with parameter $\lambda$ if
and only if the equality
%\begin{equation}\label{Equ:Poisson criterion}
$\mathbb E[Zf(Z)]=\lambda\esp[f(Z+1)]$
%\end{equation}
holds for any function $f:\mathbb N\rightarrow\mathbb
R$ such that both sides of the equality are well defined. Based on this observation,
Chen has proposed the following discrete
Stein's equation:
\begin{equation}\label{Equ:stein'sequation}
xf(x)-\lambda f(x+1)=h(x)-\mathcal P_\lambda(h), \quad x\in\mathbb N
\end{equation}
where $\mathcal P_\lambda(h)$ is the expectation of $h$
with respect to the $\lambda$-Poisson distribution. If
$X$ is an $\mathbb N $-r.v., one has
$\esp[h(X)]-\mathcal
P_\lambda(h)=\esp[Xf_h(X)-f_h(X+1)]$ where $f_h$ is a
solution of \eqref{Equ:stein'sequation} and  is given
as
\begin{equation}\label{equ:solution stein}
f_h(x)=\frac{(x-1)!}{\lambda^x}\sum_{i=x}^\infty\frac{\lambda^i}{
i!}\big(h(i)-\mathcal P_\lambda(h)\big).\end{equation}
The value $f_h(0)$ can be arbitrary and is not used in calculations in general.

Stein's method has been adopted for Poisson
approximation problems since \cite{Chen1975} in a
series of papers such as \cite{AGG1989},
\cite{BCC1995}, \cite{BarCek2002} among many others,
one can also consult the monograph \cite{BHJ1992} and
the survey paper \cite{Er2005}. In particular, Barbour
\cite{Ba1987} has developed, in parallel with the
normal case \cite{Ba1986},  asymptotic expansions for
sum of independent $\mathbb N$-r.v.s and for
polynomially growing functions. The asymptotic
expansion problem has also been studied by using other
methods such as Lindeberg method (e.g. \cite{BR2002}).

In this paper, we address this problem by the zero bias transformation approach. Similar as in Goldstein
and Reinert \cite{GR1997}, we introduce  a
discrete analogue of zero bias transformation (see also \cite{EJ2007}).
Let $X$ be an
$\mathbb N$-r.v.
with expectation $\lambda$. We say that an
$\mathbb N$-r.v. $X^*$ has Poisson
$X$-zero biased distribution if the equality
\begin{equation}\label{zero bias}
\esp[Xf(X)]=\lambda\esp[f(X^*+1)]\end{equation} holds
for any function $f:\mathbb N\rightarrow\mathbb R$ such
that the left side of \eqref{zero bias} is well
defined. The distribution of $X^*$ is unique: one has
$\mathbb P(X^*=x)=(x+1)\mathbb P(X=x+1)/\lambda$.
Combining Stein's equation \eqref{Equ:stein'sequation}
and zero bias transformation \eqref{zero bias}, the
error of the Poisson approximation can be written as
\begin{equation}\label{equ:}\mathbb E[h(X)]-\mathcal
P_\lambda(h) = \lambda\esp[f_h(X^*+1)-f_h(X+1)].\end{equation} A first order correction term for the Poisson approximation has been proposed in \cite{EJ2007} by using the Poisson zero bias transformation.

Recall the difference operator $\Delta$ defined as $\Delta f(x)=f(x+1)-f(x)$. For any
$x\in\mathbb N_*:=\mathbb N\setminus\{0\}$ and any $n\in\mathbb N$, one has
$\Delta\binom{x}{n}=\binom{x}{n-1}$. If $f$ and $g$ are two functions on $\mathbb N$,
then
%\begin{equation}\label{Equ:Delta de produt}
$$\Delta(f(x)g(x))=f(x+1)\Delta
g(x)+g(x)\Delta f(x).$$
%\end{equation}
We have the Newton's expansion (\cite[Thm5.1]{BarCek2002}), which can be viewed as an analogue of the Taylor's expansion in the discrete case.
For all $x, y\in\mathbb N$ and $N\in\mathbb N$,
\[
f(x+y)=\sum_{j=0}^N\binom{y}{j}\Delta^jf(x)+
\sum_{0\leq
j_1<\cdots<j_{N+1}<y}\Delta^{N+1}f(x+j_1).\]

Let us introduce the following quantity, where we use the same notation as in \cite{gauss},  but its meaning is changed. For any $\mathbb N$-r.v. $Y$ and any
$k\in\mathbb N$ such that $\esp[|Y|^k]<+\infty$, denote
by \begin{equation}\label{equ:m Y k}m_Y^{(k)}:=\esp\Big[\binom{Y}{k}\Big]=k!\,[Y]_k\end{equation} where $[Y]_k$ is the $k^{\text{th}}$ factorial moment of $Y$. Let $X$ and
$Y$ be two independent $\mathbb N$-r.v.s and
$f:\mathbb N\rightarrow\mathbb R$ such that
$\Delta^kf(X)$ and $\Delta^kf(X+Y)$ are both
integrable, then
\begin{equation}\label{Equ:discrete taylor formula}
\esp[f(X+Y)]=\sum_{k=0}^Nm_Y^{(k)}
\esp[\Delta^kf(X)]+\delta_N(f,X,Y).
\end{equation}
where
\begin{equation}\label{Equ:calcul de Delta}\delta_N(f,X,Y)=\esp\Big[\sum_{0\leq
j_1<\cdots<j_{N+1}<Y}
\Delta^{N+1}f(X+j_1)\Big].\end{equation}

We introduce the discrete reverse Taylor formula. Once
again, the following result is very similar with
\cite[Pro1.1]{gauss}, however, with different
significations of notation.
\begin{Pro}\label{Lem:reverse taylor}{\bf(discrete reverse Taylor formula)}
With the above notation, we have
\begin{equation}\label{Equ:reverse taylor fomrul}
\esp[f(X)]=\sum_{d\geq
0}(-1)^d\hspace{-4mm}\sum_{\mathbf{J}\in\mathbb
N^d_*,\, |\mathbf{J}|\leq N
}\hspace{-4mm}m_Y^{(\mathbf{J})}\esp[\Delta^{|\mathbf{J}|}
f(X+Y)]+ \varepsilon_N(f,X,Y)
\end{equation}
where
\begin{equation}\label{Equ:relation between esp and dleta}
\varepsilon_N(f,X,Y):=-\sum_{d\geq 0}(-1)^d
\hspace{-4mm}\sum_{\mathbf{J}\in\mathbb N^d_*,\,
|\mathbf{J}|\leq N
}\hspace{-4mm}m_Y^{(\mathbf{J})}\delta_{N-|\mathbf{J}|}(\Delta^{|\mathbf{J}|}
f,X,Y),
\end{equation}for any integer $d\geq 1$ and any
$\mathbf{J}=(j_l)_{l=1}^d\in\mathbb N_*^d$,
$|\mathbf{J}|=j_1+\cdots+j_d$ and
$m_Y^{(|\mathbf{J}|)}:=m_Y^{(j_1)}\cdots
m_Y^{(j_d)}$, and by convention, $\mathbb
N_*^0=\{\emptyset\}$ with $|\emptyset|=0$,
$m_Y^{(\emptyset)}=1$.
\end{Pro}

Consider now a family of independent $\mathbb N$-r.v.s $X_i \,(i=1,\cdots n)$ with expectations $\lambda_i$,
which are ``sufficiently good'' in a sense we shall precise later.
Let $W=X_1+\cdots+X_n$ and denote
$\lambda_W:=\esp[W]=\lambda_1+\cdots+\lambda_n$.
Let $W^{(i)}=W-X_i$
and $X_i^*$ be an
$\mathbb N$-r.v., independent of $W^{(i)}$ and
which has the Poisson  $X_i$-zero biased
distribution. Finally, let $I$ be a random index valued in
$\{1,\cdots,n\}$ which is independent of
$(X_1,\cdots,X_n,X_1^*,\cdots,X_n^*)$ and such that
$\proba(I=i)=\lambda_i/\lambda_W$ for any $i$. Then, similar as in \cite{GR1997}, the random variable $W^*:= W^{(I)}+X_I^*$ follows the Poisson $W$-zero biased
distribution.

We give below the asymptotic expansion formula in the Poisson case.

\begin{Thm}\label{Thm:main theorem}
Let $N\in\mathbb N$ and $p\geq 0$. Let $h:\mathbb
N\rightarrow\mathbb R$ be a function which is of
$O(x^p)$ at infinity and $X_i\,(i=1,\cdots,n)$ be a
family of independent $\mathbb N$-r.v.s having up to
$(N+p+1)^{\text{th}}$ order moments. Let
$W=X_1+\cdots+X_n$ and $\lambda_W=\esp[W]$. Then
$\esp[h(W)]$ can be written as the sum of two terms
$C_N(h)$ and $e_N(h)$ such that $C_0(h)=\mathcal
P_{\lambda_W}(h)$ and $e_0(h)=\esp[h(W)]-\mathcal
P_{\lambda_W}(h)$, and recursively for any $N\geq 1$,
\begin{equation}\label{Equ:recursive CN}
C_N(h)=C_0(h)+\sum_{i=1}^n\lambda_i\sum_{d\geq
1}(-1)^{d-1}\hspace{-5mm}\sum_{\mathbf{J}\in\mathbb
N_*^d,\, |\mathbf{J}|\leq N}\hspace{-4mm}
m_{X_i}^{(\mathbf J^\circ)}\big(m_{X_i^*}^{(\mathbf
J^\dagger)}-m_{X_i}^{(\mathbf{J}^\dagger)}\big)
C_{N-|\mathbf{J}|}(\Delta^{|\mathbf{J}|}f_h(x+1)),
\end{equation}
\begin{equation}\label{Equ:recursive eN}\begin{split}
e_N(h)&=\sum_{i=1}^n\lambda_i\Big[ \sum_{d\geq
1}(-1)^{d-1}\hspace{-5mm}\sum_{\mathbf{J}\in\mathbb
N_*^d,\,|\mathbf{J}|\leq
N}\hspace{-4mm}m_{X_i}^{(\mathbf
J^\circ)}\big(m_{X_i^*}^{(\mathbf
J^\dagger)}-m_{X_i}^{(\mathbf{J}^\dagger)}\big)e_{N-|\mathbf{J}|}
(\Delta^{|\mathbf{J}|}f_h(x+1))\\& \qquad
+\sum_{k=0}^Nm_{X_i^*}^{(k)}\,\varepsilon_{N-k}(\Delta^kf_h(x+1),W^{(i)},X_i)
+\delta_N(f_h(x+1),W^{(i)},X_i^*)\Big],
\end{split}\end{equation}
where for any integer $d\geq 1$ and any
$\mathbf{J}\in\mathbb N_*^d$, $\mathbf{J}^\dagger\in
\mathbb N_*$ denotes the last coordinate of
$\mathbf{J}$, and $\mathbf{J}^\circ$ denotes the
element in $\mathbb N_*^{d-1}$ obtained from
$\mathbf{J}$ by omitting the last coordinate.
\end{Thm}
\begin{Rem} In view of the similarity between the above theorem and \cite[Thm1.2]{gauss}, which has also been shown by the two papers \cite{Ba1986,Ba1987} of Barbour, the following question arises naturally: can we generalize the result to any infinitely divisible distribution?  \end{Rem}

%%%%%%%%%%%%%%%%%%
\section{Several preliminary results}
%%%%%%%%%%%%%%%%%%
In this section, we are interested in some properties
concerning the function $h$ and the associated function
$f_h$. Compared to the normal case, we no longer need
differentiability conditions on $h$  in Theorem
\ref{Thm:main theorem} and shall concentrate on its
increasing speed at infinity. This makes the study much
simpler.

We begin by considering the
modified Stein's equation on $\mathbb N_*$:
\begin{equation}\label{Equ:modified stein's equation}
x\widetilde f(x)-\lambda\widetilde
f(x+1)=h(x),\qquad x\in\mathbb N_*.
\end{equation}
The above equation may have many solutions, one of which
is given by
\[\widetilde f_h(x):=\frac{(x-1)!}{\lambda^x}\sum_{i=x}^\infty\frac{\lambda^i}{
i!}h(i).\] A general solution of
\eqref{Equ:modified stein's equation} can be
written as $\widetilde
f_h(x)+C(x-1)!/\lambda^x$, where $C$ is an
arbitrary constant. However, when $h$ is of
polynomial increasing speed at infinity,
$\widetilde f_h$ is the only solution of
\eqref{Equ:modified stein's equation} which
has polynomial increasing speed at infinity.

In order that the function $\widetilde f_h$
is well defined, we need some condition on
$h$. Denote by $\mathcal E_\lambda$ the space
of functions $h$ on $\mathbb N_*$ such that,
for any polynomial $P$, we have
\[\sum_{i\geq
1}\frac{\lambda^i}{i!}|h(i)P(i)|<+\infty.\]
Clearly $\mathcal E_\lambda$ is a linear
space. We list below some properties of $\mathcal E_\lambda$.
\begin{Pro}\label{Pro:proprietes de Elambda} The following assertions hold:
\begin{enumerate}[1)]
\item for any $Q\in\mathbb R[x,x^{-1}]$ and
any $h\in\mathcal E_\lambda$ where $\mathbb R[x,x^{-1}]$ denotes the set of Laurent polynomials on $\mathbb R$, we have $Qh\in\mathcal
E_\lambda$;
\item for any $h\in\mathcal E_\lambda$, $\Delta h\in\mathcal E_\lambda$ and
$\widetilde f_h\in\mathcal E_\lambda$;
\end{enumerate}
\end{Pro}
\begin{proof}
1) is obvious by definition.\\
2) Let $h_1$ be the function defined as
$h_1(x):=h(x+1)$. If $P$ is a polynomial, then
\[\sum_{i\geq 1}\frac{\lambda^i}{i!}\big|
P(i)h_1(i)\big|=\sum_{j\geq
2}\frac{\lambda^{j-1}}{(j-1)!}\big|P(j-1)h(j)\big|
=\lambda^{-1}\sum_{j\geq
2}\frac{\lambda^j}{j!}\big|jP(j-1)h(j)\big|<+\infty
\]
since $h\in\mathcal E_\lambda$. Therefore, $\Delta
h\in\mathcal E_\lambda$. We next prove the second
assertion. For any arbitrary polynomial $P$, there
exists another polynomial $Q$ such that, for any
integer $i\geq 1$,
$Q(i)\geq\sum_{j=1}^i\frac{|P(j)|}{j}$. Therefore
\[\begin{split}&\quad\;\sum_{a\geq 1}\frac{\lambda^a}{a!}
\Big|P(a)\frac{(a-1)!}{\lambda^a}
\sum_{i=a}^\infty\frac{\lambda^i}{i!}h(i)\Big|\leq\sum_{a\geq
1}\frac{|P(a)|}{a}\sum_{i=a}^\infty
\frac{\lambda^i}{i!}|h(i)|\\
&=\sum_{i\geq 1}\Big(\sum_{a=1}^i\frac{|P(a)|}{a}\Big)
\frac{\lambda^i}{i!}|h(i)| \leq\sum_{i\geq
1}Q(i)\frac{\lambda^i}{i!}|h(i)|<+\infty,
\end{split}\]which implies that $\widetilde f_h\in\mathcal E_\lambda$.
\end{proof}

For any function $h\in\mathcal E_\lambda$,
we define $\tau(h):\mathbb
N\rightarrow\mathbb R$ such that
$$\tau(h)(x)=h(x+1)/(x+1).$$ Note that for any
integer $k\geq 1$, one has
$\tau^k(h)(x)=x!h(x+k)/(x+k)!$. The proof of Proposition
\ref{Pro:proprietes de Elambda} shows that $\tau$ is
actually an endomorphism of $\mathcal E_\lambda$.

\begin{Lem}\label{Pro:difference of Steins equation}
Let $h\in\mathcal E_\lambda$.
Then
\begin{equation}\label{Equ:ptauhx}
\widetilde f_{\tau(h)}(x)=\widetilde f_h(x+1)/x.
\end{equation}
\end{Lem}
\begin{proof} Let $u(x)=\widetilde f_h(x+1)/x$.
Dividing both sides of \eqref{Equ:modified
stein's equation} by $x$ and then replacing $x$
by $x+1$, we obtain $\widetilde
f_h(x+1)-\lambda{\widetilde
f_h(x+2)}/(x+1)=\tau(h)(x)$, or equivalently,
\begin{equation}\label{Equ:Stein's equation for u}xu(x)-\lambda
u(x+1)=\tau(h)(x).\end{equation} Since $\widetilde
f_{\tau(h)}$ is the only solution of \eqref{Equ:Stein's
equation for u} in $\mathcal E_\lambda$, the
lemma is proved.
\end{proof}

\begin{Cor}\label{Cor:estimation de ph}
Let $p\in\mathbb R$. If
$h(x)=O( x^p)$, then $\widetilde
f_h(x)=O( x^{p-1})$.
\end{Cor}
\begin{proof}
First of all,
\[0\leq \frac{x!}{\lambda^x}\sum_{i\geq x}
\frac{\lambda^i}{i!}=\sum_{i\geq
x}\frac{\lambda^{i-x}}{i!/x!}\leq\sum_{i\geq
x}\frac{\lambda^{i-x}}{ (i-x)!}=e^{-\lambda}.\]
Therefore, when $p\leq 0$, one has
\[x\widetilde f_h(x)=\frac{x!}{\lambda^x}\sum_{i=x}^\infty
\frac{\lambda^i}{i!}h(i)=O( x^p)\]since $i^p\leq x^p$
if $x\leq i$. Hence $\widetilde f_h(x)=O( x^{p-1})$.
The general case follows by induction on $p$ by using
\eqref{Equ:ptauhx}.
\end{proof}

We now introduce the function space:
for any $p\in\mathbb R$, denote by $\mathcal H_p$
the space of all functions $h:\mathbb
N\rightarrow\mathbb R$ such that $h(x)=O(x^{p})$
when $x\rightarrow\infty$. In the following are some
simple properties of $\mathcal H_p$, their proofs are direct.

\begin{Pro}\label{Pro:basic properties}
\begin{enumerate}[1)]
\item For any $p\geq 0$ and any $h\in\mathcal H_p$,
the restriction of $h$ on $\mathbb N_*$ lies
in $\bigcap_{\lambda>0} \mathcal E_\lambda$.
\item If $h\in\mathcal H_p$, then also are
$h(x+1)$ and $\Delta h$.
\item If $h\in\mathcal H_p$ and $g\in\mathcal
H_q$, then $gh\in\mathcal
H_{p+q}$.
%\item If $P$ is a polynomial of degree $d$,
%then $P\in\mathcal H_d$.
\end{enumerate}
\end{Pro}

The following proposition is essential for applying the recursive estimation procedure.

\begin{Pro}\label{Lem:fundamental lemma}
Let $p\geq 0$. If $h\in\mathcal H_p$, then
$f_h\in\mathcal H_{p}$.
\end{Pro}
\begin{proof}
Note that $f_h$ coincides with $\widetilde f_{\overline
h }$ on $\mathbb N_*$ where $\overline h=h-\mathcal
P_\lambda(h)$. Since $h\in\mathcal H_p$, also is
$\overline h$. Then Corollary \ref{Cor:estimation de
ph} implies $f_h(x)=O(x^{p-1})=O(x^{p})$.
\end{proof}

%%%%%%%%%%%%%%%%%%%
\section{Proof of the main result}
%%%%%%%%%%%%%%%%%%%
In this section, we give the proof of
Proposition \ref{Lem:reverse taylor} and of Theorem \ref{Thm:main theorem}, which are essentially the same with the ones of \cite[Prop1.1, Thm1.2]{gauss} in a discrete setting.

\def\skip{
\begin{proof}[Proof of
We proceed by induction on $N$. When $N=0$, the
assertion follows from the equality
\[f(x+y)-f(x)=\sum_{0\leq j<y}\Delta f(x+j).\]
Assume that the assertion has been proved for
$0,\cdots,N-1$ ($N\geq 1$). Then
\[\begin{split}
f(x+y)&=\sum_{j=0}^{N-1}\binom{y}{j}\Delta^jf(x)+\sum_{0\leq
j_2<\cdots<j_{N+1}<y }\Delta^Nf(x+j_2)\\
&=\sum_{j=0}^{N-1}\binom{y}{j}\Delta^jf(x)+\sum_{0\leq
j_2<\cdots<j_{N+1}<y }\Big(\Delta^Nf(x)+\sum_{0\leq
j_1<j_2 }\Delta^{N+1}f(x+j_1)\Big)\\
&=\sum_{j=0}^N\binom{y}{j}\Delta^jf(x)+\sum_{0\leq
j_1<\cdots<j_{N+1}<k}\Delta^{N+1}f(x+j_1).
\end{split}\]
\end{proof}
}

\begin{proof}[Proof of Proposition \ref{Lem:reverse taylor}]
We replace $\esp[\Delta^{|\mathbf{J}|}f(X+Y)]$ on
the right side of \eqref{Equ:reverse taylor fomrul} by
\[\sum_{k=0}^{N-|\mathbf{J}|}m_Y^{(k)}\esp[\Delta^{|\mathbf J|+k}f(X)]
+\delta_{N-|\mathbf{J}|}(\Delta^{|\mathbf{J}|}f,X,Y)\]
and observe that the sum of terms containing $\delta$
vanishes with $\varepsilon_N(f,X,Y)$. Hence the right
side of \eqref{Equ:reverse taylor fomrul} equals
\[\sum_{d\geq
0}(-1)^d\hspace{-4mm}\sum_{\mathbf{J}\in\mathbb
N^d_*,\, |\mathbf{J}|\leq N
}\hspace{-4mm}m_Y^{(\mathbf{J})}\sum_{k=0}^{N-|\mathbf{J}|}
m_Y^{(k)}\esp[\Delta^{|\mathbf{J}|+k}f(X)]
\]
If we split the terms for $k=0$ and for $1\leq k\leq
N-|\mathbf{J}|$ respectively, the above formula can be
written as
\begin{equation}\label{Equ:proof of retf}\sum_{d\geq
0}(-1)^d\hspace{-4mm}\sum_{\mathbf{J}\in\mathbb
N^d_*,\, |\mathbf{J}|\leq N
}\hspace{-4mm}m_Y^{(\mathbf{J})}\esp[\Delta^{
|\mathbf{J}|}f(X)] +\sum_{d\geq
0}(-1)^d\hspace{-4mm}\sum_{\mathbf{J}\in\mathbb
N^d_*,\, |\mathbf{J}|\leq N
}\hspace{-4mm}m_Y^{(\mathbf{J})}\sum_{k=1}^{N-|\mathbf{J}|}
m_Y^{(k)}\esp[ \Delta^{|\mathbf{J}|+k}f(X)].
\end{equation}
We make the index changes
$\mathbf{J}'=(\mathbf{J},k)$ and $u=d+1$ in the second
part of \eqref{Equ:proof of retf} and find that it is nothing but
\[\sum_{u\geq 1}(-1)^{u-1}\hspace{-4mm}
\sum_{\mathbf{J'}\in\mathbb N^u_*,\, |\mathbf{J'}|\leq
N
}\hspace{-4mm}m_Y^{(\mathbf{J'})}\esp[\Delta^{|\mathbf{J'}|}f(X)].\]
By taking the sum, it only remains the term of index
$d=0$ in the first part of \eqref{Equ:proof of retf}, which is equal to $\esp[f(X)]$. So the lemma is proved.
\end{proof}

\begin{proof}[Proof of Theorem \ref{Thm:main theorem}] We prove the theorem by
induction on $N$. The case where $N=0$ is trivial.
Assume that the
assertion holds for $0,\cdots,N-1$.\\
Since $h\in\mathcal H_p$, by
Lemma \ref{Lem:fundamental lemma} and
Proposition \ref{Pro:basic properties} 2),
for any $k\in\{1,\cdots,N\}$,
$\Delta^{k}f_h(x+1)\in\mathcal
H_{p-1}\subset\mathcal H_p$.
Therefore $C_{N-k}(\Delta^{k}f_h(x+1))$ and
$e_{N-k}(\Delta^{k}f_h(x+1))$ are well
defined and
\[\esp[\Delta^{k}f_h(W+1)]=C_{N-k}(\Delta^{k}f_h(x+1))+e_{N-k}(\Delta^{k}f_h(x+1)).\]
We now prove the equality $\esp[h(W)]=C_N(h)+e_N(h)$. Recall that for
any $i\in\{1,\cdots,n\}$, $X_i^*$ follows the Poisson $X_i$-zero biased distribution
and is independent of $W^{(i)}=W-X_i$, $I$ is an independent random index such that
$\proba(I=i)=\lambda_i/\lambda_W$, and $W^*=W^{(I)}+X_I^*$.
So $\esp[h(W)]-C_0(h)$ is equal to
\[\lambda_W\esp[f_h(W^*+1)-f_h(W+1)]=\sum_{i=1}^n
\lambda_i\Big(\esp[f_h(W^{(i)}+X_i^*+1)]-\esp[f_h(W+1)]\Big),\]
where, by using \eqref{Equ:discrete taylor formula},
\[\esp[f_h(W^{(i)}+X_i^*+1)]=\sum_{k=0}^Nm_{X_i^*}^{(k)}
\esp[\Delta^kf_h(W^{(i)}+1)]+\delta_N(f_h(x+1),W^{(i)},X_i^*).\]
By replacing $\esp[\Delta^kf_h(W^{(i)}+1)]$ in the
above formula by its $(N-k)^{\mathrm{th}}$ order
reverse Taylor expansion, we obtain that
$\esp[f_h(W^{(i)}+X_i^*+1)]$ equals
\[\sum_{k=0}^Nm_{X_i^*}^{(k)}\bigg[\sum_{d\geq 0}
(-1)^d\hspace{-2mm}\sum_{\begin{subarray}{c}\mathbf{J}\in\mathbb
N_*^d\\|\mathbf{J}|\leq
N-k\end{subarray}}\hspace{-2mm}m_{X_i}^{(\mathbf{J})}
\esp[\Delta^{|\mathbf{J}|+k}f_h(W+1)]+\varepsilon_{N-k}
(\Delta^{k}f_h(x+1),W^{(i)},X_i)\bigg]+
\delta_N(f_h(x+1),W^{(i)},X_i^*).\] Note that the term
of indices $k=d=0$ in the sum is $\esp[f_h(W+1)]$.
Therefore, $\esp[f_h(W^{(i)}+X_i^*+1)]-\esp[f_h(W+1)]$
is the sum of the following three terms
\begin{gather}\label{Equ:first part}
\sum_{k=1}^N m_{X_i^*}^{(k)}\sum_{d\geq 0}
(-1)^d\hspace{-4mm}\sum_{\mathbf{J}\in\mathbb
N_*^d,\,|\mathbf{J}|\leq
N-k}\hspace{-4mm}m_{X_i}^{(\mathbf{J})}
\esp[\Delta^{|\mathbf{J}|+k}f_h(W+1)],\\
\label{Equ:second part}\sum_{d\geq 1}
(-1)^d\hspace{-4mm}\sum_{\mathbf{J}\in\mathbb
N_*^d,\,|\mathbf{J}|\leq
N}\hspace{-4mm}m_{X_i}^{(\mathbf{J})}
\esp[\Delta^{|\mathbf{J}|}f_h(W+1)],\\
\label{Equ:third part}
\sum_{k=0}^Nm_{X_i^*}^{(k)}\varepsilon_{N-k}(\Delta^kf_h
(x+1),W^{(i)},X_i) +\delta_N(f_h(x+1),W^{(i)},X_i^*).
\end{gather}
By interchanging summations and then making the index
changes $\mathbf{K}=(\mathbf{J},k)$ and $u=d+1$, we
obtain
\[\eqref{Equ:first part}=\sum_{u\geq 1}
(-1)^{u-1}\hspace{-4mm}\sum_{\mathbf{K}\in\mathbb
N_*^u,\,|\mathbf{K}|\leq
N}\hspace{-4mm}m_{X_i}^{(\mathbf{K^\circ})}m_{X_i^*}^{(\mathbf{K}^\dagger)}
\esp[\Delta^{|\mathbf{K}|}f_h(W+1)].
\] As the equality $m_{X_i}^{(\mathbf{J})}=m_{X_i}^{(\mathbf{J}^\circ)}
m_{X_i}^{\mathbf{J}^\dagger}$ holds for any
$\mathbf{J}$, \eqref{Equ:first part}+\eqref{Equ:second
part} simplifies as
\[\sum_{d\geq 1}
(-1)^{d-1}\hspace{-4mm}\sum_{\mathbf{J}\in\mathbb
N_*^d,\,|\mathbf{J}|\leq
N}\hspace{-4mm}m_{X_i}^{(\mathbf{J^\circ})}
\Big(m_{X_i^*}^{(\mathbf{J}^\dagger)}-m_{X_i}^{(\mathbf{J}^\dagger)}\Big)
\esp[\Delta^{|\mathbf{J}|}f_h(W+1)].\] By the
hypothesis of induction, we have
\[\esp[\Delta^{|\mathbf{J}|}f_h(W+1)]=C_{N-|\mathbf{J}|}(
\Delta^{|\mathbf{J}|}f_h(x+1))+e_{N-|\mathbf{J}|}(
\Delta^{|\mathbf{J}|}f_h(x+1)),\] so the equality
$\esp[h(W)]=C_N(h)+e_N(h)$ holds with $C_N(h)$ and $e_N(h)$ being defined in
\eqref{Equ:recursive CN} and \eqref{Equ:recursive eN}.

\end{proof}

%%%%%%%%%%%%%%%%%%%%%%%%%%%
\section{Error estimations}
%%%%%%%%%%%%%%%%%%%%%%%%%%%
In this section, we concentrate on the remainder
$e_N(h)$ in the asymptotic expansion. The following
quantity will be useful. Let $p\geq 0$. For
$h\in\mathcal H_p$ and $N\in\mathbb N $, we define
\begin{equation}\label{norm}\|h\|_{N,p}:=\sup_{x\in\mathbb N_*}\frac{\big|\Delta^{N+1}h(x)\big|}
{x^p},\end{equation} which is finite by Proposition
\ref{Pro:basic properties} 2).
\begin{Lem}
Let $N\in\mathbb N$, $k\in\{0,\cdots,N\}$ and $p\geq
0$. Let $X$ be an $\mathbb N$-r.v. with $p^{\text{th}}$
order moment, $Y$ be an $\mathbb N$-r.v. independent of
$X$ and having $(N-k+1+p)^{\text{th}}$ order moment.
Then, for any $f\in\mathcal H_p$, the following
inequalities hold:
\begin{equation}\label{Equ:estimation de delta}
|\delta_{N-k}(\Delta^kf(x+1),X,Y)|\leq \max(2^{p-1},1)
\|f\|_{N,p} \big(\mathbb E[X^p]m_Y^{(N-k+1)} +
m_Y^{(N-k+1),p}\big),\end{equation} where
$$\displaystyle m_Y^{(N-k+1),p}:=\mathbb
E\Big[\binom{Y}{N-k+1}Y^p\Big].$$ The discrete reverse
Taylor remainder satisfies
\begin{equation}\label{Equ:estimation de epsilon}\begin{split}
&\quad\big|\varepsilon_{N-k}(\Delta^kf(x+1),X,Y)\big|
\\&\leq \max(2^{p-1},1) \|f\|_{N,p}\sum_{d\geq
0}\,\sum_{\begin{subarray} {c} \mathbf{J} \in\mathbb
N^d_*\\|\mathbf{J}|\leq
N-k\end{subarray}}m_{Y}^{(\mathbf{J})}\big(\mathbb
E[X^p]m_Y^{(N-k-|\mathbf{J}|+1)}+
m_Y^{(N-k-|\mathbf{J}|+1),p}\big)\end{split}.
\end{equation}
\end{Lem}
\begin{proof}
By definition \eqref{Equ:calcul de Delta} and \eqref{norm},
\[\begin{split}&\quad\big|\delta_{N-k}(\Delta^kf(x+1),X,Y)\big|\leq
\esp\big[\sum_{0\leq
j_1<\cdots<j_{N-k+1}<Y}\big|\Delta^{N+1}f(X+1+j_1)\big|\big]\\
&\leq \|f\|_{N,p}\,\esp\big[\sum_{0\leq
j_1<\cdots<j_{N-k+1}<Y}(X+j_1+1)^p\big] \leq
\|f\|_{N,p}\,\mathbb
E\Big[\binom{Y}{N-k+1}(X+Y)^p\Big]\\
&\leq \max(2^{p-1},1) \|f\|_{N,p} \big(\mathbb
E[X^p]m_Y^{(N-k+1)} + m_Y^{(N-k+1),p}\big),
\end{split}\]
where we have used in the last inequality the
estimations $(X+Y)^p\leq 2^{p-1}(X^p+Y^p)$ if $p>1$ and
$(X+Y)^p\leq X^p+Y^p$ if $p\leq 1$. Thus
\eqref{Equ:estimation de delta} is proved. The
inequality \eqref{Equ:estimation de epsilon} follows
from \eqref{Equ:relation between esp and dleta} and
\eqref{Equ:estimation de delta}.
\end{proof}

\begin{Pro}\label{Thm:estimation} Let $N\in\mathbb N$, $p\geq 0$ and $
h\in\mathcal H_p$. Let $X_i\,(i=1,\cdots, n)$ be a
family of independent $\mathbb N$-r.v.s  with mean
$\lambda_i>0$ and up to $(N+p+1)^{\text{th}}$ order
moments; $W=X_1+\cdots+X_n$. Let $X_i^*$ be an $\mathbb
N$-r.v. having Poisson $X_i$-zero biased distribution
and independent of $W^{(i)}:=W-X_i$. Then the following
estimations hold.
\begin{enumerate}[1)]\item
When $N=0$,
\begin{equation}\label{Equ:estimation de
e0} |e_0(h)|\leq \max(2^{p-1},1)\|f_h\|_{0,p}
\sum_{i=1}^n\Big(
\esp[(W^{(i)})^p]\big(\esp[X_i^2]+\lambda_i^2-\lambda_i\big)
+\lambda_i\big(\esp[(X_i^*)^{p+1}]+\esp[(X_i)^{p+1}]\big)\Big).
\end{equation}
\item When $N\geq 1$, one has the
recursive estimation:
\begin{equation}\label{Equ:recursive esti}
\begin{split}
|e_N&(h)|\leq\sum_{i=1}^n\lambda_i\bigg[ \sum_{d\geq
1}\sum_{\substack{ \mathbf{J}\in\mathbb N_*^d\\
|\mathbf{J}|\leq N }}m_{X_i}^{(\mathbf
J^\circ)}\big(m_{X_i^*}^{(\mathbf
J^\dagger)}+m_{X_i}^{(\mathbf{J}^\dagger)}\big)|e_{N-|\mathbf{J}|}
(\Delta^{|\mathbf{J}|}f_h(x+1))|\\
& +\max(2^{p-1},1) \|f_h\|_{N,p}\sum_{k=0}^N
m_{X_i^*}^{(k)}\sum_{d\geq
0}\sum_{\substack{\mathbf{J}\in\mathbb
N^d_*\\|\mathbf{J}|\leq
N-k}}m_{X_i}^{(\mathbf{J})}\,\big(\mathbb
E[(W^{(i)})^p]m_{X_i}^{(N-k-|\mathbf{J}|+1)}+
m_{X_i}^{(N-k-|\mathbf{J}|+1),p}\big)\\
& +\max(2^{p-1},1) \|f_h\|_{N,p}\, \big(\mathbb
E[(W^{(i)})^p]m_{X_i^*}^{(N+1)} +
m_{X_i^*}^{(N+1),p}\big)\bigg],
\end{split}\end{equation}\end{enumerate}
\end{Pro}
\begin{proof}
We begin by the case when $N=0$. By \eqref{equ:},
\[\begin{split}e_0(h)&=\sum_{i=1}^n\lambda_i\big(
\mathbb E[f_h(W^{(i)}+X_i^*+1)]-\mathbb
E[f_h(W^{(i)}+X_i+1)] \big)\\&=
\sum_{i=1}^n\lambda_i\big(\delta_0(f_h(x+1),W^{(i)},X_i^*)
+\varepsilon_0(f_h(x+1),W^{(i)},X_i)\big)
\\&\leq \max(2^{p-1},1)\|f_h\|_{0,p}
\sum_{i=1}^n\lambda_i\Big\{\mathbb
E[(W^{(i)})^p]\big(m_{X_i^*}^{(1)}+
m_{X_i}^{(1)}\big)+\big(m_{X_i^*}^{(1),p}+
m_{X_i}^{(1),p}\big)\Big\}
\end{split}\]
where the last inequality is by estimations
\eqref{Equ:estimation de delta} and
\eqref{Equ:estimation de epsilon}, so
\eqref{Equ:estimation de e0} follows. Combining in
addition the recursive formula \eqref{Equ:recursive
eN}, we obtain the inequality \eqref{Equ:recursive
esti}.
\end{proof}

\bibliography{jiao}
\bibliographystyle{plain}
\end{document}